\documentclass[12pt]{article}
\usepackage{graphicx}
\usepackage{bbm}
\usepackage{amsmath, amssymb, amsthm}
\usepackage[margin=1in]{geometry}
\newtheorem{theorem}{Theorem}
\newtheorem{lemma}{Lemma}
\newtheorem{proposition}{Proposition}
\newtheorem{corollary}{Corollary}
\newtheorem{remark}{Remark}

\newtheorem{problem}{Problem}

\usepackage[dvipsnames]{xcolor} % Access to more color names
\usepackage{hyperref}

\hypersetup{
    colorlinks=true,       % False: boxed links; True: colored links
    linkcolor=MidnightBlue,    % Color of internal links (sections, etc.)
    citecolor=ForestGreen,     % Color of bibliographical citations
    filecolor=magenta,         % Color of file links
    urlcolor=MidnightBlue,     % Color of external links (URLs)
    pdftitle={Velocity Reconstruction from Flow-Induced Magnetic Fields}, 
}

\title{Velocity Reconstruction from Flow-Induced Magnetic Fields}

\author{
  Yacine Mokhtari \\ \small mokhtari.yacine10@gmail.com
  \and 
  Christina Frederick\thanks{Department of Mathematical Sciences, NJIT, Newark, NJ 07102, USA} \\ \small christin@njit.edu
  \and 
  Yunan Yang\thanks{Department of Mathematics, Cornell University, Ithaca, NY 14853, USA} \\ \small yunan.yang@cornell.edu
  \and 
  Bj\"orn Engquist\thanks{Department of Mathematics and Oden Institute, UT Austin, Austin, TX 78712, USA} \\ \small engquist@oden.utexas.edu
}

\date{\today}

\begin{document}

 \maketitle
\begin{abstract}

We study the inverse problem of reconstructing an incompressible velocity field $\boldsymbol{v}$ from observations of the induced magnetic field $\boldsymbol{b}$.
 In the presence of a strong, constant background field $\mathbf{F}$, the evolution of the magnetic perturbation $\boldsymbol{b}$ is governed by the linearized induction equation. We analyze the system on both the entire space $\Omega = \mathbb{R}^d$ and a periodic domain $\Omega = \prod_{i=1}^d [0, L_i)$, which models a homogeneous medium with side lengths $L_i > 0$. We analyze this problem by decomposing it into the evaluation of a source term from the observed field and inversion of the steady transport operator. On the whole space $\mathbb{R}^d$, we show that the transport subproblem is well-posed when data is prescribed on a non-characteristic hypersurface  transverse to $\mathbf{F}$. On the torus, we establish a sharp uniqueness criterion based on the rational relations among the scaled components $\{F_i/L_i\}_{i=1}^d$. Furthermore, we show that for the reconstructed velocity to belong to $L^2$, a sufficient condition is that the scaled background field must satisfy a Diophantine condition and that the source has corresponding Sobolev regularity.  %The reconstruction separates into evaluation of the source term from the observed field and the inversion of a steady transport operator along $\mathbf{F}$.
\end{abstract}

	%\maketitle

	\section{Introduction}

	The movement of a conducting fluid through an external magnetic field induces electric currents, which in turn generate secondary magnetic fields governed by the magnetic induction equation.  This process, known as magneto-induction, plays a crucial role in various domains, including space physics~\cite{ruderman2010stability,parks2004space}, geophysics~\cite{weaver1965magnetic,podney1975electromagnetic}, astrophysics~\cite{oughton2017reduced,gomez2005mhd}, and industrial applications~\cite{gregory2016magnetohydrodynamic,dzelme2018numerical}.  
	
	Direct velocity measurements in conducting fluids are often impractical due to opacity, extreme temperatures, or inaccessibility, such as in the Earth's core or the solar interior~\cite{wicht2010theory,jones2010solar}. An alternative approach is to infer fluid velocity from magnetic field perturbations using external measurements, such as satellite observations. This contactless method is widely employed in industrial and geophysical applications~\cite{stefani2004contactless,lehtikangas2016reconstruction,evans2022internal}.

	The inverse problem of reconstructing velocity fields from magnetic field measurements has been considered in \cite{lehtikangas2016reconstruction,stefani1999velocity,stefani2000uniqueness,stefani2012inverse}. It involves determining the internal fluid velocity using observations of the induced magnetic field. This problem plays a key role in dynamo theory, where external magnetic field measurements are used to infer fluid motion in planetary cores and astrophysical bodies~\cite{stefani2012inverse}. In oceanography, magnetometers deployed in specific regions capture magnetic field variations, and these measurements can be used for fluid velocity reconstruction~\cite{evans2022internal}. In \cite{stefani1999velocity},  the well-established methods from magnetoencephalography (MEG) \cite{hamalainen1993magnetoencephalography,fokas1996inversion} were adapted to velocity reconstruction in conducting fluids.

	These problems are related to several well-known problems that suffer from nonuniqueness, e.g., the MEG problem. Under certain assumptions, such as spherical domains, a vertically aligned background magnetic field, and stationarity (i.e., time-independent equations for small Reynolds numbers), the problem suffers from nonuniqueness. However, the incorporation of kinetic energy minimization as a regularization term has been shown to restore uniqueness~\cite{stefani2000uniqueness,stefani2012inverse} subject to no-slip boundary conditions for the velocity on the sphere. Despite this theoretical improvement, practical applications introduce additional complexities, since observed magnetic fields may originate from velocity fields that do not necessarily minimize kinetic energy, potentially leading to inaccuracies in the reconstruction process. 

    \subsection{Problem formulation}
  
	We consider the inverse problem of reconstructing the velocity field $\boldsymbol{v}$ of a conducting fluid from full internal observations of the magnetic field in dimensions \( d \in\{ 2, 3\} \). 
	The full dynamics of the magnetic field \( \boldsymbol{B} \) are governed by the magnetic induction equation
	\begin{equation}
		{{} \frac{\partial \boldsymbol{B}}{\partial t} 
		= \eta  \Delta\boldsymbol{B} + \nabla\times\left(\boldsymbol{v} \times \boldsymbol{B} \right), 
		\qquad \nabla \cdot\boldsymbol{B} = 0,}
		\label{eq:original_system}
	\end{equation}
	posed on a domain $\Omega\subseteq \mathbb{R}^d$.
    
    Here $\eta=\frac{1}{\mu_0 \sigma}$ denotes the magnetic diffusivity, where \( \sigma > 0 \) and \( \mu_0 > 0 \) denote the electrical conductivity and magnetic permeability, respectively, 
	and $t\in (0,T)$, for some $T>0$. 
	The term \( \boldsymbol{v} \times \boldsymbol{B} \) accounts for the electromagnetic induction generated by the motion of the fluid. Our primary objective is to determine the conditions under which the velocity $\boldsymbol{v}$ is uniquely identifiable from the knowledge of the induced field $\boldsymbol{B}$ within $\Omega$.

		In this work, we study the linearized model arising from the assumption that the total magnetic field is a small perturbation of a known, steady background field. Specifically, we consider the decomposition
		\[
		\boldsymbol{B} = \mathbf{F} + \boldsymbol{b},
		\]
		where \( \mathbf{F}\in\mathbb{R}^d \) is a constant background field and \( \boldsymbol{b} \) is a time-dependent perturbation with \( \| \boldsymbol{b} \| \ll \| \mathbf{F} \| \). This regime arises naturally in laboratory and geophysical settings, where weak magnetic responses are generated by slow flows in the presence of a strong imposed field; see, e.g.,~\cite{ruderman2010stability,weaver1965magnetic,stefani1999velocity}. 
		Under this assumption, linearizing~\eqref{eq:original_system} yields
		\begin{equation}
			{{}\partial_{t}\boldsymbol{b} = \eta \nabla^2\boldsymbol{b} + \nabla \times (\boldsymbol{v}\times \mathbf{F} ), \qquad \nabla \cdot \boldsymbol{b} = 0,}
			\label{eq:linearized_induction}
		\end{equation}
		a linear parabolic system in \( \boldsymbol{b} \), driven by the velocity-dependent source {{} \( \nabla \times (\boldsymbol{v}\times \mathbf{F}  )\)}.

		The inverse problem of interest in this work is:
		
		\begin{problem}
			\label{problem:inverse}  Under what conditions can a velocity field $\boldsymbol{v}$ be uniquely reconstructed from magnetic  field measurements  \( \boldsymbol{b}(t,\boldsymbol{x}) \)  taken on $\Omega $, over a time interval $[0,T]$ with $T>0$?
			
		\end{problem}

In general, this inverse problem is not uniquely solvable due to the structure of the differential operator
\[
{{} \mathcal{T}(\boldsymbol{v}) := \nabla \times (\boldsymbol{v}\times \mathbf{F} ).}
\]
Indeed, for any sufficiently regular scalar function $\phi$, the velocity fields 
$\boldsymbol{v}_1$ and $\boldsymbol{v}_2 = \boldsymbol{v}_1 + \phi\,\mathbf{F}$ satisfy
{{} \[
\mathcal{T}(\boldsymbol{v}_2)
= \nabla \times (\boldsymbol{v}_2\times \mathbf{F} )
= \nabla \times (\boldsymbol{v}_1\times \mathbf{F} )
  + \nabla \times (\phi\,\mathbf{F} \times \mathbf{F})
= \mathcal{T}(\boldsymbol{v}_1),
\]}
since $\mathbf{F} \times \mathbf{F} = \mathbf{0}$. 
Hence, the induced magnetic perturbation $\boldsymbol{b}$ is insensitive to velocity components parallel to $\mathbf{F}$. Additional structural and compatibility conditions are required to guarantee existence and uniqueness.

		\textbf{Contribution}: In this work, we provide a solution to Problem~\ref{problem:inverse} in both the case $\Omega = \mathbb{R}^d$ and the periodic setting where $\Omega$ is the $d$-dimensional torus of side lengths $\boldsymbol{L}=(L_1, ..., L_d)$, denoted by
	\begin{equation*}
		\Omega =\mathbb{T}^d_{\boldsymbol{L}}.\label{eq:periodic_domain}
	\end{equation*}

   In the case $\Omega = \mathbb{R}^d$, we provide conditions under which the solution operator for \eqref{eq:linearized_induction} uniquely determines the velocity field $\boldsymbol{v}$ given trace measurements prescribed on a non-characteristic hypersurface.
    
    On the torus $\Omega = \mathbb{T}_{\boldsymbol{L}}^d$, we show that uniqueness of the velocity field from magnetic field measurements taken over \( \Omega \) and any time interval \( [0,T] \) holds if and only if the scaled components \( \frac{F_i}{L_i} \), \( i=1,\dots,d \), are rationally independent). Furthermore, we show that for the reconstructed velocity to belong to $L^2$, the scaled background field must satisfy a Diophantine condition, and the source must possess corresponding Sobolev regularity.
		
		The remainder of the paper is structured as follows. 
			Section~\ref{sec:forwardmodel} establishes the well-posedness of the forward problem. Section~\ref{Section:induction_subproblem} shows that the kernel of the operator $\mathcal{T}(\boldsymbol{v})$ is characterized under a trace condition in $\mathbb{R}^d$ and under rational independence and a mean-zero constraint on the torus. Section~\ref{sec:inverseproblem} is devoted to the inverse problem, where we state and prove the existence and uniqueness results.

\subsection{Notation and preliminaries}In this section, we establish the functional analytic framework and notation employed throughout this work. Let $\Omega \subseteq \mathbb{R}^d$ be an open set and $T > 0$. We denote by $L^2(\Omega; \mathbb{R}^d)$ the Hilbert space of square-integrable vector fields and by $H^k(\Omega; \mathbb{R}^d)$ the standard Sobolev spaces of order $k$. For the periodic setting, $H^k_{\mathrm{per}}(\mathbb{T}^d_{\boldsymbol{L}}; \mathbb{R}^d)$ represents the space of $k$th-order Sobolev functions satisfying periodic boundary conditions on the torus $\mathbb{T}^d_{\boldsymbol{L}}$. To ensure the coercivity of the elliptic operator and to eliminate the kernel of the Laplacian (the constant functions), we restrict our analysis to subspaces of mean-zero functions:
\begin{equation}L^2_{\mathrm{per},0}(\Omega; \mathbb{R}^d) = \left\{ \boldsymbol{u} \in L^2_\mathrm{per}(\Omega; \mathbb{R}^d) : \int_{\Omega} \boldsymbol{u}(\boldsymbol{x})  d\boldsymbol{x} = \boldsymbol{0} \right\},\end{equation}
\begin{equation}H^k_{\mathrm{per},0}(\Omega; \mathbb{R}^d) = \left\{ \boldsymbol{u} \in H^k_\mathrm{per}(\Omega; \mathbb{R}^d) : \int_{\Omega} \boldsymbol{u}(\boldsymbol{x})  d\boldsymbol{x} = \boldsymbol{0} \right\}.\end{equation}
Within these subspaces, the Poincaré inequality holds, ensuring that the seminorm $\| \nabla \boldsymbol{u} \|_{L^2}$ is equivalent to the full $H^1$-norm.

% For time-dependent problems, given a Banach space $E$, $L^2(0,T; E)$ denotes the Bochner space of measurable functions $u: [0,T] \to E$ such that $\int_0^T \|u(t)\|_E^2 \, dt < \infty$. The space $C([0,T]; E)$ denotes the space of continuous functions from $[0,T]$ into $E$.

% The velocity fields in this work belong to the local $L^2$ space, defined as follows:

% The space $L^2_{\mathrm{loc}}(\mathbb{R}^d; \mathbb{R}^d)$ consists of measurable functions $\boldsymbol{v}: \mathbb{R}^d \to \mathbb{R}^d$ such that for every compact subset $K \subset \mathbb{R}^d$, the restriction $\boldsymbol{v}|_K$ belongs to $L^2(K; \mathbb{R}^d)$. 

% Let $S \subset \mathbb{R}^d$ be a smooth $(d-1)$-dimensional hypersurface. We denote by $\mathcal{D}(S)$ the space of smooth functions with compact support on $S$ (test functions). The space of distributions on $S$, denoted by $\mathcal{D}'(S)$, is defined as the dual space of $\mathcal{D}(S)$.

% For any $\boldsymbol{v} \in L^2_{\mathrm{loc}}(\mathbb{R}^d; \mathbb{R}^d)$, the divergence $\nabla \cdot \boldsymbol{v}$ is understood in the distributional sense.

\section{Well-posedness of the linearized induction problem }\label{sec:forwardmodel}

Fix the magnetic diffusivity $\eta > 0$. The linearized induction system \eqref{eq:linearized_induction} can be viewed as a forced heat equation for the magnetic perturbation $\boldsymbol{b}$:
\begin{equation}
	\left\{
	\begin{array}{ll}
		\partial_t \boldsymbol{b} - \eta \nabla^2 \boldsymbol{b}
		= \boldsymbol{h}(t,\boldsymbol{x}), 
		& (t,\boldsymbol{x}) \in (0,T)\times\Omega, \\[1ex]
		\boldsymbol{b}(0,\boldsymbol{x}) = \boldsymbol{b}_0(\boldsymbol{x}), 
		& \boldsymbol{x}\in\Omega.
	\end{array}
	\right.
	\label{eq:forwardproblem}
\end{equation}

To analyze the well-posedness of \eqref{eq:forwardproblem}, we define the following Hilbert spaces depending on the domain $\Omega$:
\begin{enumerate}\item \textbf{Whole space:} If $\Omega = \mathbb{R}^d$, we define $H^k = H^k(\mathbb{R}^d; \mathbb{R}^d)$ for $k \in \{0, 1, 2\}$.\item \textbf{Periodic torus:} If $\Omega = \mathbb{T}^d_{\boldsymbol{L}}$, we consider the Sobolev spaces of mean-zero periodic functions:$$ H^k = H^k_{\mathrm{per}}(\mathbb{T}^d_{\boldsymbol{L}}; \mathbb{R}^d) \cap L^2_{\mathrm{per},0}(\mathbb{T}^d_{\boldsymbol{L}}; \mathbb{R}^d). $$\end{enumerate}

When the initial data $\boldsymbol{b}_0$ and the source term $\boldsymbol{h}$ satisfy the conditions
$$
\boldsymbol{b}_0 \in H^1 \quad \textrm{and} \quad \boldsymbol{h} \in L^2(0,T;H^0),
$$
standard parabolic theory guarantees the existence and uniqueness of a solution $\boldsymbol{b}$ to \eqref{eq:forwardproblem} satisfying strong regularity conditions (see, for example, Chapter 7 of \cite{evans2022partial}).
\begin{equation}
\boldsymbol{b} \in L^2(0,T;H^2) \cap C([0,T];H^1),\label{eq:b_reg}    
\end{equation}
with the time derivative:
\begin{equation}
\partial_t \boldsymbol{b} \in L^2(0,T;H^0).\label{eq:bt_reg}    
\end{equation}

The following lemma states that if $ \boldsymbol{b}_0$ and $ \boldsymbol{h}$ are divergence-free, then so is $\boldsymbol{b}$.

\begin{lemma}\label{lem:fwdproblemdivfree} Let $\eta>0$ and suppose $\boldsymbol{b}_0 \in H^1$ and $\boldsymbol{h} \in L^2(0,T;H^0)$. Let $\boldsymbol{b}$ be  the unique solution of \eqref{eq:forwardproblem} satisfying \eqref{eq:b_reg} and \eqref{eq:bt_reg}. Suppose that the initial data and forcing are divergence-free:
  \[
    \nabla\cdot \boldsymbol{b}_0 = 0
    \quad\textrm{and}\quad
    \nabla\cdot \boldsymbol{h} = 0,
  \]
  in $\mathcal D'(\Omega)$. Then
  \[
    \nabla\cdot \boldsymbol{b}(\cdot,t) = 0
    \quad\textrm{for all } t\in[0,T],
  \]
\end{lemma}

\begin{proof} We present the proof for $\Omega = \mathbb{R}^d$;   similar calculations follow on the torus $\Omega = \mathbb{T}_{\boldsymbol{L}}^d$. Set $\theta(t,\boldsymbol{x}) := \nabla\cdot \boldsymbol{b}(t, \boldsymbol{x})$. Since $\boldsymbol{b} \in L^2(0,T; H^2)$, its divergence $\theta$ belongs to the scalar Sobolev space $L^2(0,T; H^1)$. Furthermore, the time continuity of $\boldsymbol{b}$ implies $\theta \in C([0,T]; L^2(\Omega))$. Since $\partial_t \boldsymbol{b} \in L^2(0,T; L^2(\Omega; \mathbb{R}^d))$, we have that $\partial_t \theta = \nabla \cdot (\partial_t \boldsymbol{b}) \in L^2(0,T; H^{-1}(\Omega))$.

  Taking the divergence of the heat equation \eqref{eq:forwardproblem} and commuting derivatives ($\nabla \cdot \nabla^2 = \nabla^2 \nabla \cdot$), we find that $\theta$ satisfies the scalar evolution equation
  \[\partial_t \theta - \eta \nabla^2 \theta = \nabla \cdot \boldsymbol{h} \quad \text{ in }  (0,T)\times\Omega.\]
  Using the assumptions $\nabla \cdot \boldsymbol{h} = 0$ and $\theta(0,\cdot) = \nabla \cdot \boldsymbol{b}_0 = 0$, this reduces to the homogeneous heat equation:
  \[
    \partial_t \theta - \eta \nabla^2 \theta = 0.
  \]
  We test this equation with $\theta$ to derive an energy estimate:
  \[
    \frac{1}{2} \frac{d}{dt} \|\theta\|_{L^2(\Omega)}^2 + \eta \int_\Omega \nabla \theta \cdot \nabla \theta \, d\boldsymbol{x}=0.
  \]

  Integrating from $0$ to $t$ implies \[\frac{1}{2}\|\theta(t,\cdot)\|_{L^2(\Omega)}^2+\eta\int_0^t\|\nabla \theta(s,\cdot)\|^2_{L^2(\Omega)} ds  = \frac{1}{2}\|\theta(0,\cdot)\|_{L^2(\Omega)}^2 = 0.\] Since $\eta>0$, this implies that $\theta(t,\cdot) \equiv 0$ almost everywhere for all $t \in [0,T]$.

\end{proof}

\begin{corollary}\label{eq:vimpliesh_reg}
Let $\mathbf{F} \in \mathbb{R}^d$ be a constant vector field and let $\boldsymbol{v} \in L^2(0,T; H^1)$. Suppose the source term in \eqref{eq:forwardproblem} is given by 
  \[
    {{} \boldsymbol{h} := \nabla \times (\boldsymbol{v}\times \mathbf{F} ).}
  \]
  Then $\boldsymbol{h} \in L^2(0,T;H^0)$ and $\nabla \cdot \boldsymbol{h} = 0$ in $\mathcal{D}'(\Omega)$.  If $\boldsymbol{b}_0 \in H^1$ and $\nabla \cdot \boldsymbol{b}_0 = 0$, then the unique strong solution $\boldsymbol{b}$ to the system \eqref{eq:forwardproblem} satisfies 
      \[
        \nabla \cdot \boldsymbol{b}(t, \cdot) = 0  \quad \textrm{for all } t \in [0,T].
      \]

\end{corollary}

\begin{corollary}\label{corr:divfree_existence}
Under the assumptions of Corollary~\ref{eq:vimpliesh_reg}, there exists a unique strong solution $\boldsymbol{b} \in L^2(0,T; H^2) \cap C([0,T]; H^1)$ to the constrained parabolic system:

\begin{equation}
\left\{
{{} \begin{array}{ll}
\partial_t \boldsymbol{b} - \eta \Delta \boldsymbol{b} = \nabla \times (  \boldsymbol{v}\times \mathbf{F}),
& (t,\boldsymbol{x}) \in (0,T)\times\Omega, \\[1ex]
\nabla \cdot \boldsymbol{b} = 0, & (t,\boldsymbol{x}) \in (0,T)\times\Omega, \\
\boldsymbol{b}(0,\boldsymbol{x}) = \boldsymbol{b}_0(\boldsymbol{x}),
 &\boldsymbol{x}\in\Omega.
\end{array}}
\right.
\label{eq:forwardproblem_divfree}
\end{equation}    
\end{corollary}

% \begin{remark}
%     On the torus $\Omega = \mathbb{T}_{\boldsymbol{L}}^d$, the theory of Navier--Stokes equations can be used to prove an analogous result to Lemma \ref{lem:fwdproblemdivfree} using divergence-free function spaces. The result is in \ref{sec:appendix}.
% \end{remark}

	\section{Existence and uniqueness of divergence-free transport solutions}\label{Section:induction_subproblem}

 In preparation to solve the inverse problem, this section provides existence and uniqueness of solutions to a transport subproblem.
 
If the magnetic perturbation $\boldsymbol{b}$ is known, the inverse problem of recovering the velocity field $\boldsymbol{v}$ is equivalent to solving the first-order linear system:
\begin{equation}
  {{} \nabla \times (\boldsymbol{v}\times \mathbf{F} ) = \boldsymbol{h},\label{eq:fnablaV}    }
\end{equation}
where the source term $\boldsymbol{h} = \partial_t \boldsymbol{b} - \eta \Delta \boldsymbol{b}$ is determined by the measurements.

When the fluid is
		incompressible (i.e., $\nabla \cdot \boldsymbol{v} = 0$), the induction term in \eqref{eq:linearized_induction} can be simplified using the vector
		identity:  {{}
		\begin{eqnarray}
			\label{idendity}
			 \nabla \times ( \boldsymbol{v}\times \mathbf{F}) &=& -(\boldsymbol{v}
			\cdot \nabla) \mathbf{F} + (\mathbf{F} \cdot \nabla) \boldsymbol{v}
			- \mathbf{F} (\nabla \cdot \boldsymbol{v}) + \boldsymbol{v} (\nabla
			\cdot \mathbf{F})   \nonumber\\
			&=& (\mathbf{F} \cdot \nabla) \boldsymbol{v}\label{eq:curl_cross}.  
		\end{eqnarray}}
Then, applying \eqref{eq:curl_cross} to \eqref{eq:fnablaV} and including the divergence-free constraint produces the system

\begin{equation}
 {{} (\mathbf{F}\cdot \nabla) \boldsymbol{v} = \boldsymbol{h}, \qquad \nabla\cdot \boldsymbol{v}=0.}\label{eq:fnablaV_divfree}    
\end{equation}
In this section, we analyze the solvability of the system \eqref{eq:fnablaV_divfree}, which constitutes the core of the velocity reconstruction problem. This system is a first-order linear transport problem where the operator $(\mathbf{F} \cdot \nabla)$ governs the propagation of information along the direction of the background field. We provide conditions for existence and uniqueness in the whole-space setting $\Omega = \mathbb{R}^d$ (Section \ref{sec:noncharacteristic}) and the periodic setting $\Omega = \mathbb{T}^d_{\boldsymbol{L}}$ (Section \ref{sec:zeromeanperiodictransport}).

The results offer a physical and geometric contrast. The operator $(\mathbf{F} \cdot \nabla)$ annihilates any function that is constant along the flow lines (characteristics) of $\mathbf{F}$. In the whole-space setting $\Omega=\mathbb{R}^d$, characteristics are infinite straight lines. Any nontrivial constant along such lines fails to be square-integrable, suggesting that solutions are uniquely determined by their values on a transverse ``inflow" section. On the torus $\Omega=\mathbb{T}^d_{\boldsymbol{L}}$, characteristics wrap around the torus. If the direction of $\mathbf{F}$ is incommensurable with the lattice, the flow is ergodic and the characteristics fill the torus densely. This recurrence forces any invariant function to be a global constant, which is then eliminated by a mean-zero constraint.

Since the operator $\mathbf{F} \cdot \nabla$ acts only on the spatial variables, we treat $t \in (0, T)$ as a parameter. Specifically, \eqref{eq:fnablaV_divfree} holds for almost every $t \in (0,T)$ in the sense of the Bochner space $L^2(0,T; L^2(\Omega; \mathbb{R}^d))$. In the following proofs, we suppress the explicit dependence on $t$ and perform the analysis on the spatial domain $\Omega$.

\subsection{Existence and uniqueness in the setting $\Omega=\mathbb{R}^d$}\label{sec:noncharacteristic}
We begin by studying ~\eqref{eq:fnablaV_divfree} on   $\Omega=\mathbb{R}^d$. Since the operator $\mathbf F\cdot\nabla$ generates translations along straight-line characteristics in the direction of $\mathbf F$, solutions are determined by their values along these characteristic curves. In order to obtain a global existence and uniqueness theory, we prescribe data on a hypersurface that intersects each characteristic exactly once. This requires the hypersurface to be transverse to the flow generated by $\mathbf F$ and to satisfy a suitable global transversality condition. Under these assumptions, the transport equation can be solved explicitly by integration along characteristics, and uniqueness follows from propagation of the data along the flow. The following theorem makes these statements precise.

\begin{theorem}[Existence and uniqueness on $\mathbb{R}^d$]
\label{thm:exist-unique-Rd}
Let $S\subset\mathbb R^d$ be a smooth $(d-1)$-dimensional hypersurface with unit normal
$\boldsymbol n(\boldsymbol x)$ such that
\[
\mathbf F\cdot \boldsymbol n(\boldsymbol x)\neq 0
\qquad \forall \boldsymbol x\in S,
\]
and assume that $S$ is globally transverse to the flow generated by $\mathbf F$, in the sense that
\[
\Phi:S\times\mathbb R\to\mathbb R^d,\qquad \Phi(y,s)=y+s\mathbf F,
\]
is a $C^1$-diffeomorphism. For $\boldsymbol x\in\mathbb R^d$, write uniquely
\[
\boldsymbol x=\pi(\boldsymbol x)+\tau(\boldsymbol x)\mathbf F,\qquad \pi(\boldsymbol x)\in S,\ \tau(\boldsymbol x)\in\mathbb R.
\]
Let $\boldsymbol h\in L^2(\mathbb R^d;\mathbb R^d)$.

\begin{enumerate}
\item[(i)]
There exists a function $\boldsymbol v\in L^2_{\mathrm{loc}}(\mathbb R^d;\mathbb R^d)$ such that {{}
\[
(\mathbf F\cdot\nabla)\boldsymbol v=\boldsymbol h\,.
\]}
Moreover, for any $\boldsymbol v_S\in L^2(S;\mathbb R^d)$, the function
\[
\boldsymbol v(\boldsymbol x)=\boldsymbol v_S(\pi(\boldsymbol x))+\int_{0}^{\tau(\boldsymbol x)}\boldsymbol h(\pi(\boldsymbol x)+s\mathbf F)\,ds
\]
defines a distributional solution with trace $\boldsymbol v|_S=\boldsymbol v_S$.

\item[(ii)]
For each prescribed trace $\boldsymbol v_S\in L^2(S;\mathbb R^d)$, there is at most one
$\boldsymbol v\in L^2_{\mathrm{loc}}(\mathbb R^d;\mathbb R^d)$ such that
\[
(\mathbf F\cdot\nabla)\boldsymbol v={{} \boldsymbol h},
\qquad
\boldsymbol v|_S=\boldsymbol v_S.
\]
\item[(iii)]
If in addition $\nabla\cdot\boldsymbol h=0$ on $\mathbb{R}^d$ and $\nabla\cdot\boldsymbol v= 0$ on $S$,
then we have $\nabla\cdot\boldsymbol v=0$ on $\mathbb{R}^d$.
\end{enumerate}
\end{theorem}

\begin{proof}
By assumption, $\Phi:S\times\mathbb R\to\mathbb R^d$, $\Phi(y,t)=y+t\mathbf F$, is a $C^1$-diffeomorphism.
Let $\Phi^{-1}(\boldsymbol x)=(\pi(\boldsymbol x),\tau(\boldsymbol x))$, so that every $\boldsymbol x\in\mathbb R^d$ admits the unique representation
$\boldsymbol x=\pi(\boldsymbol x)+\tau(\boldsymbol x)\mathbf F$ with $\pi(\boldsymbol x)\in S$ and $\tau(\boldsymbol x)\in\mathbb R$.

We first establish part \textit{(i)}. Define $\boldsymbol v$ by choosing any trace datum, for instance
$\boldsymbol v_S\equiv 0$, and setting
\begin{equation}\label{eq:v_def_Rd}
\boldsymbol v(\boldsymbol x):={{} \int_0^{\tau(\boldsymbol x)} \boldsymbol h(\pi(\boldsymbol x)+s\mathbf F)\,ds.}
\end{equation}
For each fixed $y\in S$, consider the function $\boldsymbol w_y:\mathbb R\to\mathbb R^d$ given by
$\boldsymbol w_y(s):=\boldsymbol v(y+s\mathbf F)$. By \eqref{eq:v_def_Rd},
\[
\boldsymbol w_y(s)={{} \int_0^s \boldsymbol h(y+s'\mathbf F)\,ds',}
\]
so $\boldsymbol w_y$ is absolutely continuous and satisfies
\[
\frac{d}{ds}\boldsymbol w_y(s)={{} \boldsymbol h(y+s\mathbf F)\,,}
\]
for a.e.~$s\in\mathbb{R}$. Since $\mathbf F$ is constant, the chain rule gives $\frac{d}{ds}\boldsymbol v(y+s\mathbf F)=(\mathbf F\cdot\nabla)\boldsymbol v(y+s\mathbf F)$
in the distributional sense, hence
\[
(\mathbf F\cdot\nabla)\boldsymbol v(\boldsymbol x)={{} \boldsymbol h(\boldsymbol x).}
\]
Moreover, $\boldsymbol v\in L^2_{\mathrm{loc}}(\mathbb R^d;\mathbb R^d)$. Indeed, for any bounded set $K\subset\mathbb R^d$,
the set $\Phi^{-1}(K)\subset S\times\mathbb R$ is bounded, so there is a constant $T_K>0$ such that $|\tau(\boldsymbol x)|\le T_K$ for $\boldsymbol x\in K$; then
Cauchy--Schwarz in $s$ yields
\[
|\boldsymbol v(\boldsymbol x)|^2
\le 2T_K \int_{-T_K}^{T_K} |\boldsymbol h(\pi(\boldsymbol x)+s\mathbf F)|^2\,ds,
\]
and integrating over $K$ and changing variables $\boldsymbol x=\Phi(y,s)$ shows $\int_K |\boldsymbol v|^2<\infty$
because $\boldsymbol h\in L^2(\mathbb R^d)$ and $\Phi^{-1}(K)$ has finite measure.

The same construction gives the stated representation for arbitrary trace data.
Given any $\boldsymbol v_S\in L^2(S;\mathbb R^d)$, define
\begin{equation}\label{eq:v_def_Rd_vs}
\boldsymbol v(\boldsymbol x):=\boldsymbol v_S(\pi(\boldsymbol x)) +\int_0^{\tau(\boldsymbol x)} \boldsymbol h(\pi(\boldsymbol x)+s\mathbf F)\,ds.
\end{equation}
Then $\boldsymbol v|_S=\boldsymbol v_S$ since $\tau(\boldsymbol x)=0$ for $\boldsymbol x\in S$, and repeating the preceding argument along each
characteristic shows that $\boldsymbol v$ satisfies $(\mathbf F\cdot\nabla)\boldsymbol v={{} \boldsymbol h}$
in $\mathcal D'(\mathbb R^d)$ and belongs to $L^2_{\mathrm{loc}}(\mathbb{R}^d; \mathbb{R}^d)$.

We now prove part \textit{(ii)}. Let $\boldsymbol v_1,\boldsymbol v_2\in L^2_{\mathrm{loc}}(\mathbb R^d;\mathbb R^d)$ satisfy
$(\mathbf F\cdot\nabla)\boldsymbol v_j={{} \boldsymbol h}$ in $\mathcal D'(\mathbb R^d)$ and
$\boldsymbol v_1|_S=\boldsymbol v_2|_S=\boldsymbol v_S$. Set $\boldsymbol w:=\boldsymbol v_1-\boldsymbol v_2$.
Then $(\mathbf F\cdot\nabla)\boldsymbol w=0$ in $\mathcal D'(\mathbb R^d)$ and $\boldsymbol w|_S=0$.
For fixed $y\in S$, the function $s\mapsto \boldsymbol w(y+s\mathbf F)$ is constant in $s$
(since its distributional derivative is zero), and vanishes at $s=0$ because $\boldsymbol w|_S=0$.
Hence $\boldsymbol w(y+s\mathbf F)=0$ for all $s$ and for a.e.\ $y\in S$, and therefore $\boldsymbol w=0$
a.e.\ in $\mathbb R^d$ by the bijectivity of $\Phi$. This proves uniqueness for the prescribed trace.

Finally, we prove part \textit{(iii)}. Assume $\nabla\cdot\boldsymbol h=0$ in $\mathcal D'(\mathbb R^d)$. Taking divergence of
$(\mathbf F\cdot\nabla)\boldsymbol v={{} \boldsymbol h}$ yields
\[
(\mathbf F\cdot\nabla)(\nabla\cdot\boldsymbol v)={{} \nabla\cdot\boldsymbol h}=0\,.
\]
Here, we used the fact that since $\mathbf F$ is constant, the divergence operator commutes with $\mathbf F\cdot\nabla$. As a result, $\nabla\cdot\boldsymbol v$ is constant along characteristics. Since it vanishes on $S$, it vanishes identically in $\mathbb R^d$, and hence
$\nabla\cdot\boldsymbol v=0$ in $\mathcal D'(\mathbb R^d)$.
\end{proof}

\subsection{Existence and uniqueness in the setting $\Omega = \mathbb{T}^d_{\boldsymbol{L}}$}\label{sec:zeromeanperiodictransport}
We now turn to the periodic setting, where the equation ${{} (\mathbf F\cdot\nabla)\boldsymbol v=\boldsymbol h}$ is posed on the torus $\Omega =\mathbb{T}^d_{\boldsymbol{L}}$. In this case, characteristics no longer escape to infinity but instead wrap around the domain, and the structure of the kernel of $\mathbf F\cdot\nabla$ is governed by arithmetic properties of the direction $\mathbf F$. As a result, solvability and uniqueness are determined by resonance conditions in Fourier space rather than by prescribing data on a transverse hypersurface.        
 
 We will assume $\boldsymbol{h}\in L^2_{\mathrm{per}}(\Omega;\mathbb{R}^d)$
and we seek mean-zero solutions $\boldsymbol{v}\in L^2_{\mathrm{per},0}(\Omega;\mathbb{R}^d)$ to guarantee uniqueness.

We define the \emph{resonant set} $\mathcal{R}$ as the set of all nonzero integer vectors orthogonal to the scaled background field:$$\mathcal{R} := \left\{ \boldsymbol{k} \in \mathbb{Z}^d \setminus \{\boldsymbol{0}\} : \mathbf{F} \cdot \boldsymbol{g}(\boldsymbol{k}) = 0 \right\},$$where $\boldsymbol{g}(\boldsymbol{k})$ denotes the dual lattice vector$$\boldsymbol{g}(\boldsymbol{k}) = \left( \frac{k_1}{L_1}, \dots, \frac{k_d}{L_d} \right).$$

Using this notation, we state the \emph{incommensurability condition} as:
\begin{equation} \label{rationality}\mathbf{F} \cdot \boldsymbol{g}(\boldsymbol{k}) \neq 0 \quad \textrm{for all } \boldsymbol{k} \in \mathbb{Z}^d \setminus \{\boldsymbol{0}\}.\end{equation}It follows immediately that the components $\{F_j / L_j\}_{j=1}^d$ are incommensurable if and only if $\mathcal{R} = \emptyset$.

Let $\mathbf{F} \in \mathbb{R}^d$ be a Diophantine constant background field with exponent $\tau\geq d-1$, such that for some $C > 0$:
\begin{equation}
 | \mathbf{F} \cdot \boldsymbol{g}(\boldsymbol{k}) | \ge \frac{C}{|\boldsymbol{k}|^\tau}, \quad \forall \boldsymbol{k} \in \mathbb{Z}^d \setminus \{\boldsymbol{0}\}. \label{eq:diophantine}  
\end{equation}

Notice that by definition the Diophantine condition \eqref{eq:diophantine} implies incommensurability \eqref{rationality}. An example of a field satisfying the Diophantine condition on the unit torus $\mathbb{T}^d$ is $\mathbf{F} = (1, \sqrt{2})$ for $d=2$ or $\mathbf{F} = (1, \sqrt[3]{2}, \sqrt[3]{4})$ for $d=3$.

Throughout this work, we employ the Fourier basis$$e_{\boldsymbol{k}}(\boldsymbol{x}) = e^{2\pi \mathrm{i} \, \boldsymbol{g}(\boldsymbol{k}) \cdot \boldsymbol{x}},$$which satisfies the orthogonality property on the torus $\Omega = \mathbb{T}^d_{\boldsymbol{L}}$.

\begin{theorem}[Existence and uniqueness on $\mathbb{T}^d_{\boldsymbol{L}}$]
\label{thm:exist-unique}
Let $\Omega=\mathbb{T}^d_{\boldsymbol{L}}$ and let $\mathbf{F} \in \mathbb{R}^d$ be a constant background field.

\begin{enumerate}
\item[(i)] 
For $\boldsymbol{h} \in L^2_{\mathrm{per},0}(\mathbb{T}^d_{\boldsymbol{L}})$, a solution $\boldsymbol{v} \in L^2_{\mathrm{per},0}(\Omega; \mathbb{R}^d)$ to \eqref{eq:fnablaV_divfree}, if it exists,  is unique if and only if $\mathcal{R}=\emptyset$, i.e., $\mathbf{F} \in \mathbb{R}^d$ satisfies the incommensurability condition~\eqref{rationality}.

\item[(ii)]  If $\mathbf{F} \in \mathbb{R}^d$ satisfies the Diophantine condition \eqref{eq:diophantine} with exponent $\tau\geq d-1$ and $\boldsymbol{h} \in H^{\tau}_{\mathrm{per},0}(\Omega; \mathbb{R}^d)$, then
Problem~\eqref{eq:fnablaV_divfree} admits a unique solution $\boldsymbol{v} \in L^2_{\mathrm{per},0}(\Omega; \mathbb{R}^d)$.

\item[(iii)] Suppose $\boldsymbol{h}\in L^2_{\mathrm{per},0}(\mathbb{T}^d_{\boldsymbol{L}})$ is divergence-free in the sense of distributions. If $\boldsymbol{v} \in L^2_{\mathrm{per},0}(\Omega; \mathbb{R}^d)$ is a solution to \eqref{eq:fnablaV_divfree} that satisfies $\hat{\boldsymbol{v}}(\boldsymbol{k}) = \boldsymbol{0}$ for all $\boldsymbol{k} \in \mathcal{R}$, then $\boldsymbol{v}$ is also divergence-free.
\end{enumerate}
\end{theorem}

\begin{proof}

We expand $\boldsymbol{v}$ and $\boldsymbol{h}$ in their respective Fourier series
\[
\boldsymbol v(\boldsymbol  x)=\sum_{\boldsymbol{k}\in\mathbb{Z}^d}\hat{\boldsymbol v}(\boldsymbol{k})\,e_{\boldsymbol{k}}(\boldsymbol  x),
\qquad
\boldsymbol h(\boldsymbol  x)=\sum_{\boldsymbol{k}\in\mathbb{Z}^d}\hat{\boldsymbol h}(\boldsymbol{k})\,e_{\boldsymbol{k}}(\boldsymbol  x),
\]
with $\hat{\boldsymbol v}(\boldsymbol{k}),\hat{\boldsymbol h}(\boldsymbol{k})\in\mathbb{C}^d$. Since $\mathbf F$ is a constant vector field,
\[
(\mathbf F\cdot\nabla)e_{\boldsymbol{k}}=2\pi \mathrm{i}\,(\mathbf F\cdot\boldsymbol{g}(\boldsymbol{k}))\,e_{\boldsymbol{k}}.
\]
Substituting into ${{} (\mathbf F\cdot\nabla)\boldsymbol v=\boldsymbol h}$ yields, for each $\boldsymbol{k}\in\mathbb{Z}^d$,
{{} \begin{equation}\label{eq:mode_eq}
2\pi \mathrm{i}\,(\mathbf F\cdot\boldsymbol{g}(\boldsymbol{k}))\,\hat{\boldsymbol v}(\boldsymbol{k})
=
\hat{\boldsymbol h}(\boldsymbol{k}).
\end{equation}}

We prove part \textit{(i)}.
The only sources of nonuniqueness correspond to the modes $\boldsymbol{k}=\boldsymbol{0}$ and $\boldsymbol{k}\in\mathcal R$. Imposing the zero-mean condition $\hat{\boldsymbol v}(\boldsymbol{0})=\boldsymbol{0}$ fixes the mean. For $\boldsymbol{k}\in\mathcal R$, Equation~\eqref{eq:mode_eq} reduces to the equation $\boldsymbol 0 (\hat{\boldsymbol v}(\boldsymbol{k})) = \hat{\boldsymbol h}(\boldsymbol{k})$, so the coefficients $\hat{\boldsymbol v}(\boldsymbol{k})$ are unconstrained. Consequently, the solution is unique in $L^2_{\mathrm{per},0}(\Omega;\mathbb{R}^d)$ if and only if $\mathcal R=\emptyset$, which is exactly the incommensurability condition~\eqref{rationality}.

We now prove part \textit{(ii)}. Here, $\mathcal R =\emptyset$.  If $\boldsymbol{k}=\boldsymbol 0$, the left-hand side of \eqref{eq:mode_eq} vanishes and the right-hand side also vanishes because
$\hat{\boldsymbol h}(\boldsymbol{k})=\boldsymbol 0$. As in (i), imposing the zero-mean condition $\hat{\boldsymbol v}(\boldsymbol{0})=\boldsymbol{0}$ fixes the mean. If  $\boldsymbol{k}\neq\boldsymbol 0$, then
$\mathbf F\cdot\boldsymbol{g}(\boldsymbol{k})\neq 0$ and \eqref{eq:mode_eq} uniquely determines
\[
{{} \hat{\boldsymbol v}(\boldsymbol{k})
=
\frac{1}{2\pi \mathrm{i}\,(\mathbf F\cdot\boldsymbol{g}(\boldsymbol{k}))}\,\hat{\boldsymbol h}(\boldsymbol{k}).}
\]

Since $\mathbf{F}$ satisfies \eqref{eq:diophantine}, we have the following estimate for the Fourier coefficients of $\boldsymbol{v}$:$$\|\boldsymbol{v}\|_{L^2(\Omega; \mathbb{R}^d)}^2 = \sum_{\boldsymbol{k} \neq \mathbf{0}} \frac{|\hat{\boldsymbol{h}}(\boldsymbol{k})|^2}{4\pi^2 |\mathbf{F}\cdot \boldsymbol{g}(\boldsymbol{k})|^2} \le \frac{1}{4\pi^2 C^2} \sum_{\boldsymbol{k} \neq \mathbf{0}} |\boldsymbol{k}|^{2\tau} |\hat{\boldsymbol{h}}(\boldsymbol{k})|^2 \le \frac{1}{4\pi^2 C^2} \|\boldsymbol{h}\|_{H^\tau(\Omega; \mathbb{R}^d)}^2.
$$

Finally, we prove part \textit{(iii)}. Assume that $\boldsymbol h$ is divergence free, which in
Fourier variables is equivalent to
\[
\boldsymbol{g}(\boldsymbol{k})\cdot\hat{\boldsymbol h}(\boldsymbol{k})=0
\qquad \forall\boldsymbol{k}\in\mathbb{Z}^d.
\]
Taking the dot product of \eqref{eq:mode_eq} with $\boldsymbol{g}(\boldsymbol{k})$ gives
{{} \[
2\pi \mathrm{i}\,(\mathbf F\cdot\boldsymbol{g}(\boldsymbol{k}))\,\boldsymbol{g}(\boldsymbol{k})\cdot\hat{\boldsymbol v}(\boldsymbol{k})
=
\boldsymbol{g}(\boldsymbol{k})\cdot\hat{\boldsymbol h}(\boldsymbol{k})
=0.
\]}
For all $\boldsymbol{k}\notin\mathcal R\cup\{\boldsymbol 0\}$ this implies
$\boldsymbol{g}(\boldsymbol{k})\cdot\hat{\boldsymbol v}(\boldsymbol{k})=0$. Since for $\boldsymbol{k}\in\mathcal R$, $\hat{\boldsymbol v}(\boldsymbol{k})=\boldsymbol 0$ and $\hat{\boldsymbol v}(\boldsymbol{0})=\boldsymbol 0$, we conclude that $\boldsymbol{v}$ is divergence free. 
\end{proof}

\begin{remark}
There is an alternative proof to part~(i) based on ergodic theory if $\boldsymbol{v}\in C(\Omega;\mathbb{R}^d)$. It is sufficient to consider the homogeneous problem
\begin{equation}
\label{homogeneous-problem}
(\mathbf{F}\cdot\nabla)\boldsymbol{v}=0,
\qquad
\nabla\cdot\boldsymbol{v}=0.
\end{equation}
The goal is to show that problem~\eqref{homogeneous-problem} admits only the
trivial solution $\boldsymbol{v}\equiv\boldsymbol{0}$ among zero-mean vector
fields.

Since $\mathbf{F}$ is constant, the characteristics of the transport operator are straight lines \[ \boldsymbol{x}(s)=\boldsymbol{x}_0 + s\mathbf{F} \qquad (\mathrm{mod}\ L_1,\dots,L_d). \] Along such a curve, the chain rule gives \[ \frac{d}{ds} v_i(\boldsymbol{x}(s)) = (\mathbf{F}\cdot\nabla)v_i(\boldsymbol{x}(s)),\quad 1\leq i \leq d\,, \] and hence any solution of the homogeneous problem~\eqref{homogeneous-problem} must satisfy 
\[ \frac{d}{ds}v_i(\boldsymbol{x}(s))=0,\quad 1\leq i \leq d\,. 
\] 
Thus each component $v_i$, $1\leq i \leq d$, is constant along the \emph{orbit} 
\[ \gamma(\boldsymbol{x}_0) := \{\, \boldsymbol{x}_0 + s\mathbf{F} \; (\mathrm{mod}\ L_1,\dots,L_d): s\in\mathbb{R}\,\}. \] 

The structure of these orbits depends on the arithmetic properties of $\mathbf{F}$. The incommensurability condition~\eqref{rationality} has a classical interpretation in dynamical systems. It is well-known that the flow \begin{equation}\label{eq:dyn-sys} 
\boldsymbol{x}\mapsto \boldsymbol{x}+s\mathbf{F} \quad (\mathrm{mod}\ L_1,\dots,L_d) \end{equation} 
is \emph{ergodic} on $\mathbb{T}^d_L$ if and only if condition~\eqref{rationality} holds. In this case, the orbit $\gamma(\boldsymbol{x}_0)$ is dense in the entire torus for every $\boldsymbol{x}_0\in\Omega$. In other words, the invariant measure of the dynamical system~\eqref{eq:dyn-sys} is the Lebesgue measure on the torus. If $\boldsymbol{v}$ is continuous and satisfies $(\mathbf{F}\cdot\nabla)\boldsymbol{v}=0$, then $v_i$ must be constant on each orbit for $1\leq i \leq d$, and hence constant on the closure of each orbit. Since all orbits are dense under the condition~\eqref{rationality}, each $v_i$ must be constant on $\Omega$. Since $\boldsymbol{v}$ is additionally mean zero, $\boldsymbol{v}\equiv \boldsymbol{0}$. Thus, the homogeneous problem~\eqref{homogeneous-problem} admits only the trivial solution.

On the other hand, if condition~\eqref{rationality} fails, then there exists $k^\star\in\mathbb{Z}^d\setminus\{\boldsymbol{0}\}$ such that \[ \sum_{i=1}^d \frac{F_i}{L_i}\,k^\star_i =0. \] In this case, the orbit $\gamma(\boldsymbol{x}_0)$ lies on a lower-dimensional subtorus and is not dense. One may construct nontrivial functions that are constant along these lower-dimensional orbits. Consequently, the homogeneous problem admits nonzero solutions.
\end{remark}
\begin{remark}
			Note that the problem \eqref{eq:fnablaV_divfree} is a particular case of the curl-div system
			\begin{equation}
				\left\{ 
				\begin{array}{ll}
					\nabla \times (\boldsymbol{A}\boldsymbol{v}) = \boldsymbol{h}, & \boldsymbol{x} \in \Omega, \\
					\nabla \cdot \boldsymbol{v} = \boldsymbol{g}, & \boldsymbol{x} \in \Omega,
				\end{array}
				\right.
				\label{singular_system}
			\end{equation}
			which has been studied under the assumption that $\boldsymbol{A}$ is uniformly positive definite \cite{alonso2019curl}. In the present setting $\boldsymbol{A}\boldsymbol{v} = {{} \boldsymbol{v}\times \mathbf{F}}$, the linear map $\boldsymbol{A}$ is singular.
            
		\end{remark}

\section{Inverse problem}\label{sec:inverseproblem}

 The main results describe when the observed field $\boldsymbol{b} $ uniquely determines a velocity field $\boldsymbol{v}$, that is, when the solution operator corresponding to \eqref{eq:linearized_induction} is injective. 
%  We define the following spaces:
% \[
% \mathbb{V}:= \{ \boldsymbol{b} \in L^2(0,T;H^2) \cap C([0,T];H^1), 
% \partial_t \boldsymbol{b} \in L^2(0,T;H^0), \nabla\cdot \boldsymbol{b}=0 \textrm{ for a.e. } t \in [0,T]
% \}\]
% \[
% \mathbb{X}:=\{\boldsymbol{v}\in L^2(0,T; H^0),\quad  \nabla\cdot \boldsymbol{v}=0 \textrm{ for a.e. } t\in[0,T]\}\]

		\begin{theorem}\label{main-theorem-rd} Let $\eta, T>0$ and $\mathbf{F}\in \mathbb{R}^d$ be a fixed constant background field. Assume that the function $\boldsymbol{b}$ satisfies
        \begin{enumerate}
            \item $\boldsymbol{b} \in L^2(0,T;H^2(\mathbb{R}^d; \mathbb{R}^d)) \cap C([0,T];H^1(\mathbb{R}^d; \mathbb{R}^d))$
            \item $\partial_t \boldsymbol{b} \in L^2(0,T;L^2(\mathbb{R}^d; \mathbb{R}^d))$
            \item $\nabla\cdot \boldsymbol{b}=0 \textrm{ for a.e. } t \in [0,T]$.
        \end{enumerate}

        Furthermore, suppose $S\subset\mathbb{R}^{d}$ is a $(d-1)-$dimensional, non-characteristic hypersurface satisfying the assumptions of Theorem \ref{thm:exist-unique-Rd} and let $\boldsymbol{v}_S\in L^2(0,T; L^2(S;\mathbb{R}^d))$ be a given function.

Then, there is a unique velocity field $\boldsymbol{v}\in \mathbb{X}:=\{\boldsymbol{v}\in L^2(0,T; L^2_{\mathrm{loc}}(\mathbb{R}^d; \mathbb{R}^d)),  \nabla\cdot \boldsymbol{v}=0 \textrm{ for a.e. } t\in[0,T]\}$ satisfying $\boldsymbol{v}\vert_S=\boldsymbol{v}_S$ such that \eqref{eq:linearized_induction} is satisfied.
          
		\end{theorem}

        \begin{proof}
Define, for $(t,\boldsymbol x)\in(0,T)\times\mathbb{R}^d$,
\[
\boldsymbol h(t,\boldsymbol x)
:= \partial_t \boldsymbol b(t,\boldsymbol x)-\eta \nabla^2 \boldsymbol b(t,\boldsymbol x).
\]
Then $\boldsymbol h\in L^2(0,T;L^2(\mathbb{R}^d; \mathbb{R}^d))$. Moreover, since $\nabla\cdot \boldsymbol b=0$ and
$\nabla\cdot\nabla^2=\nabla^2\nabla\cdot$,
\[
\nabla\cdot \boldsymbol h
=\partial_t(\nabla\cdot \boldsymbol b)-\eta \nabla^2(\nabla\cdot \boldsymbol b)=0.
\]

By Theorem~\ref{thm:exist-unique-Rd}, for a.e. $t\in[0,T]$ there exists a unique solution $\boldsymbol v(t,\cdot)\in L^2_{\mathrm{loc}}(\mathbb{R}^d; \mathbb{R}^d)$ with $\nabla\cdot\boldsymbol v(t,\cdot)=0$ and
${{} (\mathbf{F}\cdot\nabla)\boldsymbol v(t,\cdot)=\boldsymbol h(t,\cdot)\quad\textrm{in }\Omega}$ such that $\boldsymbol{v}(t,\cdot)\vert_{S} = \boldsymbol{v}_S(t,\cdot)$. Since $\boldsymbol{h}$ is divergence-free, so is $\boldsymbol{v}$. Moreover, for any compact set $K\subset \mathbb{R}^d$ we have for a.e. $t\in(0,T)$,

\[
\|\boldsymbol v(t,\cdot)\|_{L^2(K; \mathbb{R}^d)}^2 \le C_K\left( \|\boldsymbol v_S(t,\cdot)\|^2_{L^2(S; \mathbb{R}^d)}  + \|\boldsymbol h(t,\cdot)\|^2_{L^2(K; \mathbb{R}^d)}\right).
\]

Integrating over $t\in [0,T]$ and noting that $\boldsymbol v_{S}\in L^2(0,T; L^2(S;\mathbb{R}^d))$ and $\boldsymbol{h}\in L^2(0,T; L^2(\mathbb{R}^d; \mathbb{R}^d))$ we have $\int_0^T\|\boldsymbol v(t,\cdot)\|_{L^2(K; \mathbb{R}^d)}^2dt<\infty$, and therefore $\boldsymbol v\in \mathbb{X}$.

\end{proof}

The next two theorems concern the periodic case.

\begin{theorem}\label{thm:main-theorem-uniqueness-torus}
Let $\mathbf{F} \in \mathbb{R}^d$ and $\Omega = \mathbb{T}^d_{\boldsymbol{L}}$. Suppose there exists a divergence-free velocity field $\boldsymbol{v} \in L^2(0,T; L^2_{\mathrm{per},0}(\Omega; \mathbb{R}^d))$ satisfying the linearized induction equation \eqref{eq:linearized_induction} for a given magnetic field $\boldsymbol{b} \in L^2(0,T; H^2_{\mathrm{per}}(\Omega;\mathbb{R}^d)) \cap C([0,T]; H^1_{\mathrm{per}}(\Omega;\mathbb{R}^d))$. This velocity field is unique if and only if $\mathbf{F}$ is incommensurable with respect to the lattice $\boldsymbol{L}$, i.e., \eqref{rationality} is satisfied.
\end{theorem}

\begin{proof}Since $\boldsymbol{v}$ is divergence-free and $\mathbf{F}$ is constant, equation \eqref{eq:curl_cross} gives $\nabla \times (\boldsymbol v\times \mathbf F) = (\mathbf F\cdot \nabla )\boldsymbol v$. Therefore, equation \eqref{eq:linearized_induction} determines the transport equation \eqref{eq:fnablaV_divfree}.
 The result follows from Theorem~\ref{thm:exist-unique} (i).
\end{proof}

		\begin{theorem}\label{main-theorem} Let $\Omega = \mathbb{T}^d_{\boldsymbol{L}}$, and let  $\mathbf{F}$ be a Diophantine vector with exponent $\tau\geq d-1$, and let $\eta, T>0$. Assume that the function $\boldsymbol{b}$ satisfies
        \begin{enumerate}
            \item $\boldsymbol{b} \in L^2(0,T;H^{\tau+2}_{\mathrm{per}}(\Omega; \mathbb{R}^d)) \cap C([0,T];H^{\tau+1}_{\mathrm{per}}(\Omega; \mathbb{R}^d))$
            \item $\partial_t \boldsymbol{b} \in L^2(0,T;H^\tau_{\mathrm{per}}(\Omega; \mathbb{R}^d))$
            \item $\nabla\cdot \boldsymbol{b}=0 \textrm{ for a.e. } t \in [0,T]$
            \item $\int_{\Omega}\boldsymbol{b}(t,\boldsymbol{x})d\boldsymbol{x}=0$ for almost every $t\in [0,T]$.
        \end{enumerate}
Then there exists a unique velocity field $\boldsymbol{v}\in \mathbb{X}:=\{\boldsymbol{v}\in L^2(0,T; L^2_{\mathrm{per},0}(\Omega; \mathbb{R}^d)), \nabla\cdot \boldsymbol{v}=0 \textrm{ for a.e. } t\in[0,T]\}$
        satisfying \eqref{eq:linearized_induction}.
         Moreover, $\boldsymbol{v}$ satisfies the stability estimate:$$ \|\boldsymbol{v}\|_{L^2(0,T; L^2(\Omega;\mathbb{R}^d))} \le \frac{1}{2\pi C} \left\| \partial_t \boldsymbol{b} - \eta \nabla^2 \boldsymbol{b} \right\|_{L^2(0,T; H^\tau(\Omega;\mathbb{R}^d))}. $$

		\end{theorem}
		
\begin{proof}
Define, for $(t,\boldsymbol x)\in(0,T)\times\Omega$,
\[
\boldsymbol h(t,\boldsymbol x)
:= \partial_t \boldsymbol b(t,\boldsymbol x)-\eta \nabla^2 \boldsymbol b(t,\boldsymbol x).
\]
Then $\boldsymbol h\in L^2(0,T;H^{\tau}_{\mathrm{per}}(\Omega; \mathbb{R}^d))$. Moreover, since $\nabla\cdot \boldsymbol b=0$ and
$\nabla\cdot\nabla^2=\nabla^2\nabla\cdot$,
\[
\nabla\cdot \boldsymbol h
=\partial_t(\nabla\cdot \boldsymbol b)-\eta \nabla^2(\nabla\cdot \boldsymbol b)=0.
\]
Since $\int_{\Omega}\boldsymbol b(t,\boldsymbol x)\,d\boldsymbol x=0$ for a.e.\ $t\in(0,T)$, we have $\int_{\Omega}\boldsymbol h(t,\boldsymbol x)\,d\boldsymbol x=0$ for a.e.\ $t$.

By Theorem~\ref{thm:exist-unique}, for a.e. $t\in[0,T]$ there exists a unique,
zero-mean, divergence-free function $\boldsymbol v(t,\cdot)\in L^2_{\mathrm{per},0}(\Omega; \mathbb{R}^d)$ satisfying
${{} (\mathbf{F}\cdot\nabla)\boldsymbol v(t,\cdot)=\boldsymbol h(t,\cdot)\textrm{ in }\Omega}$.
Moreover, the operator $(\boldsymbol{F}\cdot \nabla)^{-1}:H^{\tau}_{\textrm{per},0}(\Omega; \mathbb{R}^d)\rightarrow L^2_{\textrm{per},0}(\Omega; \mathbb{R}^d)$ is bounded with
\[
\|\boldsymbol v(t,\cdot)\|_{L^2_{\mathrm{per}}(\Omega; \mathbb{R}^d)} \le \frac{1}{2\pi C} \|\boldsymbol h(t,\cdot)\|_{H^\tau_{\mathrm{per}}(\Omega; \mathbb{R}^d)}
\quad\textrm{for a.e. }t\in(0,T),
\]
with the constant $C$ given by \eqref{eq:diophantine}.
Since $\boldsymbol h\in L^2(0,T;H^\tau_{\mathrm{per}}(\Omega; \mathbb{R}^d))$, integrating over $t\in[0,T]$ yields $\boldsymbol v\in L^2(0,T;L^2_{\mathrm{per},0}(\Omega; \mathbb{R}^d))$,
i.e.\ $\boldsymbol v\in\mathbb X$.

\medskip
This proves conditions for which $\boldsymbol b$ determines a unique velocity field $\boldsymbol v\in\mathbb X$
satisfying ${{} (\mathbf{F}\cdot\nabla)\boldsymbol v = (\partial_t \boldsymbol b(t,\boldsymbol x)-\eta \nabla^2 \boldsymbol b(t,\boldsymbol x))}$ (equivalently, \eqref{eq:linearized_induction}).
\end{proof}

 Theorems \ref{thm:main-theorem-uniqueness-torus} and~\ref{main-theorem} demonstrate that the solvability and uniqueness of the velocity reconstruction on the torus depend fundamentally on the arithmetic relationship between the background field $\mathbf{F}$ and the domain dimensions $\boldsymbol{L}$. If the components satisfy the incommensurability condition~\eqref{rationality}, the transport direction induced by $\mathbf{F}$ never aligns with the periodic lattice of $\boldsymbol{L}$. Consequently, the resonant set $\mathcal{R}$ is empty, and the velocity field is uniquely determined. If this condition fails, the problem admits a nontrivial kernel, and uniqueness is lost. While incommensurability ensures uniqueness, the Diophantine condition \eqref{eq:diophantine} provides existence and stability by ensuring that the recovered velocity field is a well-defined $L^2$ function.

		\section{Concluding Remarks and Open Questions}

In this paper, we have provided a rigorous treatment of the existence and uniqueness of solutions to the inverse problem for the linearized magnetic induction equation, establishing conditions under which an incompressible velocity field can be uniquely reconstructed from measurements of the induced magnetic perturbation. By separating the reconstruction into the evaluation of the source term from the observed field and the inversion of a steady transport operator, we showed that the identifiability of the velocity field is governed by geometric and arithmetic properties of the background magnetic field and the spatial domain. 

Several avenues for further investigation in applied analysis are suggested by this work. Here, it is assumed that the field $\boldsymbol{b}$ is known everywhere. A natural extension would be to relax this assumption and prove conditions on uniqueness and stability of the reconstruction by imposing additional constraints on $\boldsymbol{v}$, e.g., dynamical constraints requiring $\boldsymbol{v}$ to solve a steady Stokes or Navier--Stokes system. The inclusion of the nonlinear term $\boldsymbol{v} \times \boldsymbol{b}$ represents a substantial leap in complexity. Furthermore, a natural progression of this theoretical work is the development of robust computational algorithms for velocity reconstruction in the whole-space case.

Overall, the results presented here provide a theoretical foundation for understanding when and how fluid velocities can be inferred from magnetic field observations, clarifying the interplay between geometry, transport dynamics, and inverse reconstruction in magnetohydrodynamic systems.
		%%%%%%%%%%%%%%%%%%%%%%%%%%%%%%%%%%%%%%%%%%%%%%%%%%%%%%%%%%%%%%%%%%%%%%
		
		\section*{Acknowledgment}
		
		This work was supported by ONR N00014-24-1-2095 and ONR N00014-24-1-2088. The authors would like to acknowledge insightful discussions with David Shirokoff and Haomin Zhou.
        CF also acknowledges the J. Tinsley Oden Faculty Fellowship Research Program.\\

       The authors assume responsibility for all content.

		\bibliographystyle{unsrt}
		\bibliography{biblio}

\appendix

\end{document}